\documentclass[12pt]{amsart}
\usepackage{amsmath,amsthm,amsfonts,amssymb,latexsym,enumerate} 
\usepackage{showlabels}
\usepackage[pagebackref]{hyperref} 

\newcommand{\bC}{{\mathbf C}}

\newcommand{\bO} {\mathbf O}

\newcommand{\bZ}{{\mathbf Z}}
\newcommand{\bF}{{\mathbf F}}

\newcommand{\AAA}{\mathsf{A}}

\newcommand{\Aut}{{{\operatorname{Aut}}}}

\newcommand{\Irr}{{{\operatorname{Irr}}}}
\newcommand{\GL}{\operatorname{GL}}

\newcommand{\Sp}{\operatorname{Sp}}

\newcommand{\Ker}{\operatorname{Ker}}

\newtheorem{thm}{Theorem}[section]
\newtheorem{lem}[thm]{Lemma}

\newtheorem*{thmA}{Theorem A}
\newtheorem*{conA'}{Conjecture A'}
\newtheorem*{thmB}{Theorem B}

\theoremstyle{definition}

\numberwithin{equation}{section}

\begin{document}

\title[Kernels of minimal characters]{Kernels of minimal characters of solvable groups}

\author{Alexander Moret\'o}
\address{Departamento de Matem\'aticas, Universidad de Valencia, 46100
  Burjassot, Valencia, Spain}
\email{alexander.moreto@uv.es}

\thanks{Research supported by  Ministerio de Ciencia e Innovaci\'on (Grant PID2019-103854GB-I00 funded by MCIN/AEI/ 10.13039/501100011033).}



\date{\today}

\begin{abstract}
Let $G$ be a finite solvable group. We prove that if $\chi\in\Irr(G)$ has odd degree and $\chi(1)$ is the minimal degree of the non-linear irreducible characters of $G$, then $G/\Ker\chi$ is nilpotent-by-abelian. 
\end{abstract}

\maketitle


\section{Introduction}  

A classical theorem of Broline and Garrison implies that if an irreducible character $\chi$ of a finite group $G$ has maximal degree then $\Ker\chi$ is nilpotent (Corollary 12.20 of \cite{isa}). This result was extended by Isaacs, who considered characters of $n$th maximal degree in \cite{isa09}, and proved that  if $\chi\in\Irr(G)$ has $n$th maximal degree, then the Fitting height of the solvable radical of $\Ker\chi$ is at most $n$.

Our goal in this note is to consider irreducible characters at the other extreme. Of course, if $\chi\in\Irr(G)$ is linear, then $G'\leq\Ker\chi$ and $G/\Ker\chi$ is abelian. But can we restrict the structure of $G/\Ker\chi$ if $\chi(1)$ is ``small"? This is the content of our main result. We write $m(G)=\min\{\chi(1)\mid\chi\in\Irr(G), \chi(1)>1\}$. We say that $\chi\in\Irr(G)$ is a  {\bf minimal character} if $\chi(1)=m(G)$. 

\begin{thmA}
Let $G$ be a solvable finite group. Suppose that $m(G)$ is odd. If $\chi\in\Irr(G)$ is a minimal character, then $G/\Ker\chi$ is nilpotent-by-abelian.
\end{thmA}

As $\GL_2(3)$ shows, some hypothesis on $m(G)$ is definitely necessary. The Frobenius group of order $20$ acting  faithfully on an extraspecial $2$-group of order $2^5$ is an example with faithful minimal characters of degree $4$. We do not know whether it is enough to assume that $m(G)$ is not a power of $2$. 
On the other hand, some solvability hypothesis is definitely necessary: consider any non-abelian simple group with odd degree minimal characters (for instance, $\AAA_5$). 
Theorem A follows from applying the next result to $G/\Ker\chi$.

\begin{thmB}
Let $G$ be a finite  solvable group. Suppose that  $\chi\in\Irr(G)$ is a faithful minimal character. If $\chi(1)$ is odd, then $G$ is nilpotent-by-abelian.
\end{thmB}

Note that the structure of groups with {\it all} minimal characters faithful was described in detail by Robinson in \cite{rob1, rob2}. In particular, as shown in Lemma 2.1 of \cite{rob1}, solvable groups with all minimal characters faithful are nilpotent-by-abelian.  The examples mentioned above show that this is not the case if we just assume that $G$ has a minimal faithful character. Our proof of Theorem B relies on some of the ideas developed by Robinson.

\section{Proofs}

We argue as in Lemma 2 of \cite{rob2} to prove our first lemma.

\begin{lem}
\label{rest}
Let $G$ be a finite group. Suppose that $\chi\in\Irr(G)$ is a primitive faithful minimal character of $G$. If $N\trianglelefteq G$ is non-central, then $\chi_N\in\Irr(N)$.
\end{lem}

\begin{proof}
Suppose that $\chi_N\not\in\Irr(N)$. Then there exists a central extension $G^*$ of $G$ and $\alpha,\beta\in\Irr(G^*)$ such that $\chi=\alpha\beta$, where $\alpha,\beta$ are primitive non-linear irreducible characters of $G^*$. Without loss of generality, we may assume that $\alpha(1)\leq\chi(1)^{1/2}$. Since $1_{G^*}$ is an irreducible constituent of $\alpha\overline{\alpha}$, the minimality of $\chi(1)$ implies that  $\alpha\overline{\alpha}$ is a sum of linear characters. Hence $(G^*)'\leq\Ker(\alpha\overline{\alpha})$. In particular, $(G^*)'\leq\bZ(\alpha)$.  By Lemma 2.27 of \cite{isa}, $\bZ(\alpha)/\Ker\alpha\leq\bZ(G^*/\Ker\alpha)$. If follows that $G^*/\Ker\alpha$ is nilpotent (of class at most $2$). But $\alpha$ is non-linear and primitive. This contradicts Theorem 6.22 of \cite{isa}.
\end{proof}

Next, we handle the primitive case of Theorem B. We refer the reader to \cite{isa2} for the definition and basic properties of Gajendragadkar's $p$-special characters. 

\begin{lem}
\label{prim}
Let $G$ be a solvable group. Suppose that  $\chi\in\Irr(G)$ is a primitive faithful minimal  character. Then 
$\chi(1)$ is a  power of a prime $p$. Furthermore,
if $p>2$, then $G$ is nilpotent-by-abelian.
\end{lem}

\begin{proof}
We may assume that $\chi(1)>1$.
By Theorem 2.17 of \cite{isa2}, $\chi$ factors as a product of $p$-special characters, where $p$ runs over the set of prime divisors of $\chi(1)$. Since $\chi(1)=m(G)$, it follows that $\chi(1)=p^n$ is a power of a prime $p$. This proves the first part of the lemma.

Suppose now that $p>2$. 
 Let $q\neq p$ be a prime. Then $\chi_{\bO_q(G)}$ is not irreducible. It follows from Lemma \ref{rest} that $\bO_q(G)$ is central in $G$. Hence $\bF(G)=E\bZ(G)$, where $E=\bO_p(G)$. Furthermore, using again Lemma \ref{rest}, every normal abelian subgroup of $G$ is central. Since $G$ has a faithful irreducible character, Theorem 2.32 of \cite{isa} implies that $\bZ(G)$ is cyclic. Now, Corollary 1.10 of \cite{mw} implies that $E$ is extraspecial of exponent $p$. Since $\chi_E\in\Irr(E)$, we necessarily have that $|E|=p^{2n+1}$.

Note that $\bC_G(E)=\bC_G(\bF(G))=\bZ(G)$, so $G/\bZ(G)$ is isomorphic to a subgroup of $\Aut_{\bZ(E)}(E)$. By \cite{win}, using again that $p>2$, we deduce that $G/\bF(G)$ is isomorphic to a subgroup of $\Sp(2n,p)$. By \cite{ls}, $\Sp(2n,p)$ has a faithful irreducible representation of dimension $(p^n-1)/2$. Hence $G/\bF(G)$ has a faithful character of degree $\leq (p^n-1)/2$. Since $m(G)=p^n$, we conclude that $G/\bF(G)$ has a faithful character  that is a sum of linear characters. We conclude that $G/\bF(G)$ is abelian, as we wanted to prove.
\end{proof}

Now, we complete the proof of a slightly strengthened version of Theorem B.

\begin{thm}
Let $G$ be a solvable group. Suppose that  $\chi\in\Irr(G)$ is a faithful minimal character. If $\chi$ is induced from an odd degree character, then $G$ is nilpotent-by-abelian.
\end{thm}

\begin{proof}
Let $H\leq G$ and $\beta\in\Irr(H)$ primitive such that $\beta^G=\chi$. Suppose first that $\beta(1)=1$, so that $\chi(1)=|G:H|$. Since $1_G$ is an irreducible constituent of $(1_H)^G$ and $m(G)=|G:H|=(1_H)^G(1)$, we deduce that $(1_H)^G$ is a sum of linear characters. Hence 
$G'\leq\Ker(1_H)^G\leq H$. Thus $H\trianglelefteq G$ and by Clifford's theorem (Theorem 6.2 of \cite{isa}), $\chi_H$ is a sum of conjugates of $\beta$. In particular, $\chi_H$ is a sum of linear characters, so $H'\leq\Ker\chi=1$. Hence $G$ is metabelian and the result follows.

Now, we may assume that $\beta(1)>1$ is odd. First, we will see that $H\trianglelefteq G$ and $G'=H'$. Note that $\chi(1)=|G:H|\beta(1)>|G:H|$. Hence $(1_H)^G$ is a sum of linear characters and $G'\leq H$, as before. In particular, $H\trianglelefteq G$. Thus $H'$ is also normal in $G$ and all the irreducible characters of $G/H'$ have degree divisible by $|G:H|<\chi(1)=m(G)$. Hence, $\Irr(G/H')$ is a set of linear characters, and we conclude that $G'=H'$, as desired.

Now, we claim that $\beta(1)=m(H)$. Let $\mu\in\Irr(H)$ be non-linear. Hence, there exists $\nu\in\Irr(G)$ non-linear such that $[\mu^G,\nu]\neq 0$. Thus
$$
|G:H|\beta(1)=\chi(1)=m(G)\leq\nu(1)\leq\mu^G(1)=|G:H|\mu(1).
$$
We conclude that $\beta(1)\leq\mu(1)$. The claim follows.

Thus $\beta$ is a primitive faithful minimal character of $H/\Ker\beta$. By Lemma \ref{prim}, we have that $\beta(1)$ is a power of a prime $p$.  Let $\beta=\beta_1,\dots,\beta_t$ be the $G$-conjugates of $\beta$. Let $K_i=\Ker\beta_i$. Since $\chi=\beta^G$ is faithful, Lemma 5.11 of \cite{isa} implies that $\bigcap_{i=1}^tK_i=1$. 
Since $p>2$, by the second part of Lemma \ref{prim}, we have that $H/K_i$ is nilpotent-by-abelian. Write $F_i/K_i=\bF(H/K_i)$, so that $G'=H'\leq F_i$ for every $i$. By Proposition 9.5 of \cite{mw}, $\bigcap_{i=1}^tF_i=\bF(H)$. Therefore, $G'\leq \bF(H)$ and we conclude that $G$ is nilpotent-by-abelian, as wanted.
\end{proof}


\end{document}